\documentclass[a4paper,reqno]{amsart}
\usepackage{amsmath, amsthm, amsfonts, amssymb, mathrsfs}
\usepackage{enumitem}
\usepackage{graphicx}
\usepackage[update,prepend]{epstopdf} 
\usepackage{color}

\newcommand{\N}{\mathbb{N}}
\newcommand{\R}{{\mathbb{R}}}
\newcommand{\C}{{\mathbb{C}}}

\newcommand{\ii}{{\rm i}}

\renewcommand{\H}{{\mathcal{H}}}
\newcommand{\X}{{\mathcal{X}}}
\newcommand{\G}{{\mathcal{G}}}
\newcommand{\Dom}{{\operatorname{Dom}}}

\renewcommand{\span}{{\operatorname{span}}}

\newcommand{\diag}{{\operatorname{diag}}}

\newcommand{\cf}{\emph{cf.}}
\newcommand{\ie}{{\emph{i.e.}}}
\newcommand{\eg}{{\emph{e.g.}}}
\newcommand{\dist}{\mathrm{dist}}

\newcommand{\s}{\stackrel{s}{\rightarrow}}
\newcommand{\tolong}{\longrightarrow}
\newcommand{\gnrlong}{\stackrel{gnr}{\longrightarrow}}

\newcommand{\lm}{\lambda}

\begin{document}

\graphicspath{{C:/Data/00Synchronized/Figures/2014/psp_conv/}}

\theoremstyle{plain}
\newtheorem{theorem}{Theorem}[section]
\newtheorem{lemma}[theorem]{Lemma}
\newtheorem{criterion}[theorem]{Criterion}
\newtheorem{proposition}[theorem]{Proposition}
\newtheorem{corollary}[theorem]{Corollary}
\renewcommand{\proofname}{Proof}
\theoremstyle{definition} 
\newtheorem{define}[theorem]{Definition}
\newtheorem{example}[theorem]{Example}
\newtheorem{remark}[theorem]{Remark}
\newtheorem{ass}[theorem]{Assumption}

\title{Remarks on the convergence of pseudospectra}

\author{Sabine B\"ogli}
\address[Sabine B\"ogli]{
Mathematisches Institut, 
Universit\"{a}t Bern,
Sidlerstrasse 5,
3012 Bern, Switzerland}
\email{sabine.boegli@math.unibe.ch}
\author{Petr Siegl}
\address[Petr Siegl]{Mathematisches Institut, Universit\"at Bern, Sidlerstrasse 5, 3012 Bern, Switzerland \& On leave from Nuclear Physics Institute ASCR, 25068 \v Re\v z, Czech Republic}
\email{petr.siegl@math.unibe.ch}



\date{
\today
}

\begin{abstract}
We establish the convergence of pseudospectra in Hausdorff distance for closed operators acting in different Hilbert spaces and converging in the generalised norm resolvent sense. 
As an assumption, we exclude the case that the limiting operator has constant resolvent norm on an open set. 
We extend the class of operators for which it is known that the latter cannot happen by showing that if the resolvent norm is constant on an open set, then this constant is the global minimum. 
We present a number of examples exhibiting various resolvent norm behaviours 
and illustrating the applicability of this characterisation compared to known results.

\end{abstract}

\thanks{
The first author gratefully acknowledges the support of the Swiss National Science Foundation (SNF), grant no.\ 200020\_146477. 
The second author was supported by the Swiss Scientific Exchange Programme SCIEX, project no.\ 11.263.
}

\maketitle

\section{Introduction}
For $\varepsilon>0$ the $\varepsilon$-pseudospectrum of a closed operator $T$ acting in a Banach space~$\X$ is defined as the set
\begin{equation}\label{eq.def.of.sigma.eps}
\sigma_\varepsilon(T) := \left\{
z \in \C \, : \, \|(T-z)^{-1}\| > \frac{1}{\varepsilon}
\right\},
\end{equation}
where we employ the convention that $\|(T-z)^{-1}\|=\infty$ for $z$ in the spectrum of~$T$. 
While the $\varepsilon$-pseudospectrum of a normal operator in a Hilbert space coincides with the $\varepsilon$-neighbourhood of the spectrum, the situation is more involved in the non-normal case or for operators acting in Banach spaces, \cf~\cite{Trefethen-2005, Davies-2007, Helffer-2013-book}. 
In this paper we address the convergence of pseudospectra for sequences of unbounded operators and the related problem when the resolvent norm may be constant on an open set.

It is well known that spectra do not necessarily behave well under limiting procedures, even for bounded operators converging  in operator norm, \cf~\cite[Ex.~IV.3.8]{Kato-1966}.
Stability problems are simpler when passing from spectra to pseudospectra.
Consider a sequence $\{T_k\}_k$ of operators that converges (in some sense) to an operator~$T$. One might study convergence of pseudospectra, \ie
\begin{equation}\label{eq.conv.of.pseudospec}
\lim_{k\to\infty}\sigma_{\varepsilon}(T_k)=\sigma_{\varepsilon}(T)
 \end{equation}
(where the limit is defined appropriately) or the analogous identity for the closed sets.
Such results were established for truncated Wiener-Hopf and Toeplitz operators, \cf~\cite{Boettcher-1994}, and
for constant-coefficient differential operators, \cf~\cite{Reddy-1993-5,Davies-2000-43}; many results can also be found in the books~\cite{Trefethen-2005,Boettcher-Silbermann}.
In~\cite[Prop.~4.2]{Boettcher-Wolf-1997} necessary conditions for the inclusion ``$\supseteq$'' in~\eqref{eq.conv.of.pseudospec} are given for general approximations of a bounded operator.
The identity~\eqref{eq.conv.of.pseudospec} is known to hold for a sequence of bounded operators that converges in operator norm, \cf~\cite{Harrabi-1998}.
The convergence of pseudospectra in Hausdorff distance was proved in~\cite[Thm.~5.3]{Hansen-2011-24}
for (possibly unbounded) operators acting in the same Hilbert space and converging in the gap topology. 
This convergence result extends to the generalised $\varepsilon$-pseudospectrum, the so-called $(n,\varepsilon)$-pseudospectrum introduced in \cite{Hansen-2008-254}, \cf~ \eqref{eq.def.of.sigma.n.eps} below.
All the above pseudospectral convergence results rely on the condition that the limiting operator does not have constant resolvent norm on an open set 
(as probably first noted in~\cite{Boettcher-1994}); for classes of operators which do not a priori satisfy this condition, it needs to by guaranteed by assumption.

In view of applications in PDEs, \eg~the domain truncation method
where the operators act in different Hilbert spaces,
\cf~\cite{Brown-2004-24,Boegli-2014} and the references therein, 
we employ the so-called generalised norm resolvent convergence, \cf~Section \ref{sec:conv}.
Our first main result is the pseudospectral convergence for a sequence of operators that converges in the generalised norm resolvent sense, 
\cf~Theorem~\ref{thm.psp.conv} for the $\varepsilon$-pseudospectra and Theorem~\ref{thm.psp.conv.n} for its generalisation to $(n,\varepsilon)$-pseudospectra. 
Note that if all operators act in the same space, generalised norm resolvent convergence coincides with usual norm resolvent convergence.
If the resolvent set is non-empty, the latter convergence is equivalent to convergence in the gap topology, \cf~\cite[Thm.~IV.2.23]{Kato-1966}. Hence our results generalise ~\cite[Thm.~5.3]{Hansen-2011-24}.
Moreover, we show that if the operators all act in the same Banach space and converge in the norm resolvent sense, our pseudospectral convergence result remains valid, \cf~Remark~\ref{rem.Banach}.

Besides generalised norm resolvent convergence, two additional assumptions are needed in Theorem~\ref{thm.psp.conv} (and analogously in Theorem~\ref{thm.psp.conv.n}).
The first assumption guarantees a suitable selection of a compact set $K \subset \C$ in which the pseudospectral convergence is shown.
This assumption cannot be weakened to the one in~\cite[Thm.~5.3]{Hansen-2011-24}, see the counterexample in~Example~\ref{ex.counter.K}.
The second assumption excludes the possibility of constant resolvent norm on an open set. The necessity of this assumption is discussed in Example~\ref{ex.counter.const}.

Whether the resolvent norm of a bounded operator in a Banach space can be constant on an open set was first studied by Globevnik, \cf~\cite{Globevnik-1976-20}, see also \cite{Shargorodsky-2009-93} for a history of the problem. 
He showed that this cannot happen in the unbounded component of the resolvent set. 
Since then, the occurrence of constant resolvent norm on an open set has been excluded for various classes of Banach spaces~$\X$ and closed operators $T$ acting in $\X$, namely 
if $\X$ is complex uniformly convex and $T$ is bounded, \cf~\cite{Globevnik-1976-20,Boettcher-1997-3,Shargorodsky-2008-40}, or generates a $C_0$ semigroup, \cf~\cite{Shargorodsky-2010-42}, or
if $\X$ is complex strictly convex and $T$ is densely defined with compact resolvent, 
\cf~\cite{Davies-Shargorodsky-2014}. 
Moreover, by duality, it suffices that the condition on the Banach space is satisfied for $\X^*$. 
The definitions of complex uniform and complex strict convexity can be found in~\cite{Globevnik-1975-47,Shargorodsky-2009}, \cf~also Definition \ref{def.cuc} below. 
In particular, Hilbert spaces and $L^p$ spaces with $1 \leq p <\infty$ are complex uniformly convex (and thus complex strictly convex),  \cf~\cite{Clarkson-1936-40,Globevnik-1975-47}.

On the other hand, several examples are known in which the resolvent norm is constant on an open set,
\cf~\cite{Shargorodsky-2008-40,Shargorodsky-2009-93};
both contain examples of bounded operators in Banach spaces (that are not complex uniformly convex), the former also includes the construction of an unbounded operator in a Hilbert space.

As the second main result, we prove that if a closed operator $T$ acts in a complex uniformly convex Banach space and its resolvent norm is constant on an open set, then this constant is the global minimum, \cf~Theorem~\ref{thm.const.norm}; 
in Theorem~\ref{thm.const.norm.n} the result is generalised for higher powers of the resolvent. 
As a consequence, a resolvent norm decay (for $\lm\in\rho(T)$ tending to infinity along some path) is a sufficient condition for excluding constant resolvent norm on an open set. 
This applies in particular if $T$ is bounded or generates a $C_0$ semigroup,
\cf~Corollary~\ref{cor.bdd.or.semigroup}.
Nonetheless, Theorem~\ref{thm.const.norm} enables one to go beyond these two classes as shown in Example \ref{ex.not.semi}. 
The latter belongs to a class of examples in Hilbert spaces, \cf~the end of Section~\ref{sec:const}, that illustrates various resolvent norm behaviours and also naturally includes Shargorodsky's example \cite[Thm.~3.2]{Shargorodsky-2008-40} of an operator whose resolvent norm is constant on an open set.

The possible occurrence of constant resolvent norm on an open set is also relevant in the discussion about the definition of pseudospectra, \cf~\cite{Chatelin-Harrabi-98, Shargorodsky-2009}. 
Depending on the literature, for a closed operator $T$ and $\varepsilon>0$, both sets $\sigma_{\varepsilon}(T)$ and $\Sigma_{\varepsilon}(T)$ are called $\varepsilon$-pseudospectrum, where 
$\sigma_{\varepsilon}(T)$ is defined in~\eqref{eq.def.of.sigma.eps} and $\Sigma_{\varepsilon}(T)$ is the same with strict inequality replaced by non-strict inequality. 
This makes the set $\sigma_{\varepsilon}(T)$ an open and $\Sigma_{\varepsilon}(T)$ a closed subset of $\C$.
The closure of $\sigma_{\varepsilon}(T)$ is always contained in $\Sigma_{\varepsilon}(T)$, but equality holds if and only if $T$ does not have constant resolvent norm on any open set.

Throughout this paper, we use the following notation:
By $\|\cdot\|$ we denote the norm of all considered Banach spaces (it should be clear from the context what space is considered).
For an operator $T$ with domain $\Dom(T)$, the spectrum, approximate point spectrum and resolvent set are denoted by $\sigma(T)$, $\sigma_{\rm app}(T)$ and $\rho(T)$, respectively.
For a sequence $\{T_k\}_k$ of bounded operators acting in the same Banach space and converging strongly or in norm to an operator $T$, we write 
$T_k\stackrel{s}{\to} T$ and $T_k\to T$, respectively.
Finally, the open ball with radius $r>0$ around $\lm\in\C$ is denoted by~$B_r(\lm)$.

\section{Convergence of pseudospectra}
\label{sec:conv}

Our result on pseudospectral convergence is formulated for closed operators $T$ and $T_k$ acting in (possibly different) Hilbert spaces $\H$ and $\H_k$. A suitable convergence of operators in such a situation can be introduced following \eg~\cite{Weidmann-1997-34}. 
Assume that $\H$ and $\H_k$ are subspaces of one ``large'' Hilbert space $\G$ and $P:=P_{\H}$, $P_k:=P_{\H_k}$ are the orthogonal projections onto the respective subspaces. 
Then $T_k$ is said to converge to $T$ in the generalised
norm resolvent sense if 
\begin{equation}
\exists \lm_0 \in \bigcap_{k\in\N} \rho(T_k) \cap\rho(T) : \quad (T_k-\lm_0)^{-1}P_k \tolong (T-\lm_0)^{-1}P
\end{equation}
and we write  
$T_k \gnrlong T$.

Analogously to \cite{Hansen-2008-254,Hansen-2011-24}, we express the convergence of the $\varepsilon$-pseudospectra in the sense of Hausdorff distance which is defined for two non-empty compact sets $K_1, K_2 \subset \C$ as
\begin{equation}
d_{\rm H}(K_1,K_2):= \max 
\left\{
\sup_{z \in K_1} \dist(z,K_2), \sup_{z \in K_2} \dist(z,K_1) 
\right\},
\end{equation}
where, for $z \in \C$, $\dist(z,K_j):=\inf_{w \in K_j}|z-w|$.

The following result is proved in Section~\ref{sec:proof.1}. 
\begin{theorem}\label{thm.psp.conv}
Let $T, T_k$ be closed, densely defined operators acting in Hilbert spaces $\H,$ $\H_k$, respectively, with non-empty intersection of the resolvent sets. Let $K \subset \C$ be compact and $\varepsilon >0$. If 
\begin{enumerate}[label=\rm{\roman{*})}]
\item $\overline{\sigma_{\varepsilon}(T)} \cap K = \overline{\sigma_{\varepsilon}(T) \cap K} \neq \emptyset$,
\label{ass.K}
\item $\lambda\mapsto\|(T-\lambda)^{-1}\|$ is non-constant on any open subset of $\rho(T)$,
\label{ass.const}
\item $T_k \gnrlong T$,
\label{ass.1.3}
\end{enumerate}
then
\begin{equation}\label{dH.conv}
d_{\rm H} 
\left(
\overline{\sigma_{\varepsilon}(T_k)} \cap K, \overline{\sigma_{\varepsilon}(T)} \cap K
\right) \tolong 0,  \quad k \to \infty.
\end{equation}
\end{theorem}
\begin{remark}
\label{rem.Banach}
The conclusion of Theorem \ref{thm.psp.conv} holds also if $T$ and $T_k$ act in the same Banach space $\X$, converge in the usual norm resolvent sense and the assumptions~\ref{ass.K},~\ref{ass.const} are satisfied.
See Remark \ref{rem.Banach.2} below for the strategy of the proof in this situation.
\end{remark}

Example \ref{ex.counter.K} shows that, in general, assumption~\ref{ass.K} cannot be weakened to the condition $\overline{\sigma_\varepsilon(T)} \cap K \neq \emptyset$ 
or to the assumption $\sigma_\varepsilon(T) \cap K \neq \emptyset$ used in~\cite[Thm.~5.3]{Hansen-2011-24};
in fact, the claim of~\cite[Prop.~4.1]{Hansen-2011-24} does not hold under this weaker condition.

We remark that if $K$ is selected as a subset of $\sigma_\varepsilon(T)$, then assumption~\ref{ass.const} is not needed, \cf~Lemma \ref{lem.K}. 
However, Example~\ref{ex.counter.const} shows that assumption~\ref{ass.const} cannot be omitted in general.

\begin{example}\label{ex.counter.K}
Let $\H:=\C^2$ and $T:= \diag(\lambda_1,\lambda_2)$ with $\lambda_1 < \lambda_2$. The $\varepsilon$-pseudospectrum of $T$ is the union of the two open $\varepsilon$-balls $B_{\varepsilon}(\lambda_1)$ and $B_{\varepsilon}(\lambda_2)$. We choose $K \subset \C$ to be a rectangle that touches the first ball and contains the second ball, 
\ie~ $\overline{B_\varepsilon (\lambda_1)} \cap K = \{w_0\}$ and $B_{\varepsilon}(\lambda_2) \subset K$ as indicated in Figure \ref{fig:counter}.
As the approximating matrix, consider $T_k:= \diag ((1-1/k) \lambda_1, \lambda_2)$. Its $\varepsilon$-pseudospectrum consists also of two $\varepsilon$-balls; the first one approaches $B_{\varepsilon}(\lambda_1)$ from the left and the second one coincides with $B_{\varepsilon}(\lambda_2)$, \cf~Figure \ref{fig:counter}.
\begin{figure}[htb!]
\includegraphics[width= 0.7 \textwidth]{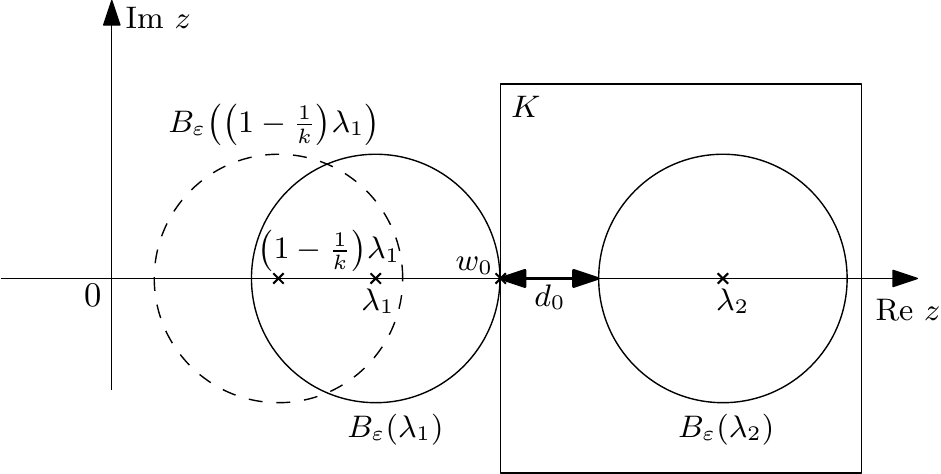}
\caption{The choice of $K$.}
\label{fig:counter}
\end{figure}

\noindent
Then the pseudospectra do not converge since
\begin{equation}
d_{\rm H} 
\left(
\overline{\sigma_{\varepsilon}(T_k)} \cap K, \overline{\sigma_{\varepsilon}(T)} \cap K
\right) 
\geq \dist\left(w_0, \overline{\sigma_{\varepsilon}(T_k)} \cap K\right)  = d_0 >0.
\end{equation}
Analogously, if we take the same $T$ and $K$, but $T_k:= \diag ((1+1/k) \lambda_1, \lambda_2)$, the first ball of $\sigma_{\varepsilon}(T_k)$ approaches $B_{\varepsilon}(\lambda_1)$ from the right, so we obtain, for all sufficiently large $k$, 
\begin{equation}
d_{\rm H} 
\left(
\overline{\sigma_{\varepsilon}(T_k) \cap K}, \overline{\sigma_{\varepsilon}(T) \cap K}
\right) \geq \dist\left(w_0, \overline{\sigma_{\varepsilon}(T) \cap K}\right) = d_0 >0.
\end{equation}
Therefore we cannot weaken assumption \ref{ass.K} to $\sigma_\varepsilon(T) \cap K  \neq \emptyset$, even with the appropriate modification in the claim.
\end{example}

\begin{example}\label{ex.counter.const}
Let $T$ be a closed operator acting in a Hilbert space~$\H$ such that $\sigma(T) \neq \emptyset$ and there exist an open set $U\subset\rho(T)$ and a constant $M>0$ such that $\|(T-\lm)^{-1}\|=M$ for all $\lm\in U$.
Such operators exist, \cf~for instance Shargorodsky's example~\cite[Thm.~3.2]{Shargorodsky-2008-40}.
Note that $U\neq\C$ since $\sigma(T)\neq\emptyset$.
We approximate $T$ by $T_k:=(1-1/k)T$; it is easy to see that $T_k \gnrlong T$.

By Theorem~\ref{thm.const.norm} below, the constant $M$ is the global minimum of $\lm\mapsto \|(T-\lm)^{-1}\|$ in~$\rho(T)$.
So we obtain, for all $k$ and all $\lm\in\rho(T_k)$,
\begin{equation}
\begin{aligned}
 \|(T_k-\lm)^{-1}\|=\left(1-\frac{1}{k}\right)^{-1}\left\|\left(T-\frac{\lm}{1-\frac{1}{k}}\right)^{-1}\right\|\geq \left(1-\frac{1}{k}\right)^{-1} M > M,
\end{aligned}
\end{equation}
hence $\sigma_{1/M}(T_k)=\C$. On the other hand, $\sigma_{1/M}(T)\cap U=\emptyset$.
Now take a compact set $K\subset\C$ such that $K\cap U\neq\emptyset$ and $K \cap \sigma(T) \neq \emptyset$, hence $K \nsubseteq U$. 
Let $z_0 \in K\cap U$. 
Then, for all $k$, we have $z_0 \in K=\overline{\sigma_{1/M}(T_k)}\cap K$, and hence
\begin{equation}
\begin{aligned}
d_{\rm H} 
\left(
\overline{\sigma_{1/M}(T_k)} \cap K, \overline{\sigma_{1/M}(T)} \cap K
\right) 
&\geq \dist\left(z_0,\overline{\sigma_{1/M}(T)}\cap K\right) \\
&\geq \dist\left(z_0, K \setminus U \right)>0.
\end{aligned}
\end{equation}
In summary, assumption \ref{ass.const} in Theorem \ref{thm.psp.conv} cannot be omitted in general.
\end{example}

For $n\in\N_0$ and $\varepsilon>0$ the $(n,\varepsilon)$-pseudospectrum of a closed operator $T$ acting in a Banach space~$\X$ is defined as the set
\begin{equation}\label{eq.def.of.sigma.n.eps}
\sigma_{n,\varepsilon}(T) := \left\{
z \in \C \, : \, \left\|(T-z)^{-2^n}\right\|^{\frac{1}{2^n}} > \frac{1}{\varepsilon}
\right\};
\end{equation}
note that for $n=0$ the notion coincides with the usual $\varepsilon$-pseudospectrum. 
The interesting property of the $(n, \varepsilon)$-pseudospectrum is that, for $n$ tending to infinity, it converges in Hausdorff distance to the $\varepsilon$-neighbourhood of the spectrum, \cf~\cite[Thm.~5.1]{Hansen-2011-24}.

The pseudospectral convergence result can be generalised for the $n$-pseudospectra; Theorem~\ref{thm.psp.conv} is the formulation for $n=0$. 
In Section~\ref{sec:proof.1} we indicate the different and additional steps that are needed for proving the following claim for $n>0$.
\begin{theorem}\label{thm.psp.conv.n}
Let $n\in\N_0$ and let $T, T_k$ be closed, densely defined operators acting in Hilbert spaces $\H,$ $\H_k$, respectively, with non-empty intersection of the resolvent sets. Let $K \subset \C$ be compact and $\varepsilon >0$. If 
\begin{enumerate}[label=\rm{\roman{*})}]
\item $\overline{\sigma_{n,\varepsilon}(T)} \cap K = \overline{\sigma_{n,\varepsilon}(T) \cap K} \neq \emptyset$,
\label{ass.K.n}
\item $\lambda\mapsto\|(T-\lambda)^{-2^n}\|$ is non-constant on any open subset of $\rho(T)$,
\label{ass.const.n}
\item $T_k \gnrlong T$,
\label{ass.1.3.n}
\end{enumerate}
then
\begin{equation}\label{dH.conv.n}
d_{\rm H} 
\left(
\overline{\sigma_{n,\varepsilon}(T_k)} \cap K, \overline{\sigma_{n,\varepsilon}(T)} \cap K
\right) \tolong 0,  \quad k \to \infty.
\end{equation}
\end{theorem}

\begin{remark}
If assumption~\ref{ass.K.n} is omitted, the operators in Example~\ref{ex.counter.K} remain 
relevant counterexamples since they are selfadjoint, hence 
$\sigma_{n,\varepsilon}(T) = \sigma_{\varepsilon}(T)$ for each $n\in\N$. 

If, for an $n>0$, assumption~\ref{ass.const.n} is omitted, an analogous counterexample as Example~\ref{ex.counter.const} can be constructed.
In order to find an operator $T$ with $\sigma(T)\neq\emptyset$  whose $2^n$-th power of the resolvent has constant norm on an open set, one may proceed as in Shargorodsky's example~\cite[Thm.~3.2]{Shargorodsky-2008-40} for $n=0$.
More exactly, let $T:={\rm diag}(B_1,B_2,B_3,\dots)$ with $2^{n+1}\times 2^{n+1}$-matrices $B_k,\,k\in\N,$ of the form 
\begin{equation}
 B_k:=\begin{pmatrix} 0 & A_k\\ \widetilde A_k & 0\end{pmatrix}.
\end{equation}
The entries of the $2^n\times 2^n$-matrices $A_k,\widetilde A_k$ are chosen in such a way that $\|(B_k-\lm)^{-2^n}\|\to 1$ as $k\to\infty$ and $\|(B_k-\lm)^{-2^n}\|<1$, $k\in\N$, for all $\lm$ in a neighbourhood of $\lm=0$.
Then $\|(T-\lm)^{-2^n}\|\equiv 1$ on this neighbourhood.
For instance for $n=1$ such an example is  $T:={\rm diag}(B_1,B_2,B_3,\dots)$ with
\begin{equation}
 B_k:=\begin{pmatrix} 0 & 0 & 0 & \beta_k\\ 0 & 0 & \alpha_k & 0\\ \alpha_k & 0 & 0 & 0\\ 0 & \beta_k & 0 & 0\end{pmatrix}, \quad k\in\N,
\end{equation}
where $\alpha_k\geq 2$, $\alpha_k\to\infty$ as $k\to\infty$, and $\beta_k:=1+1/\alpha_k$.
\end{remark}

\section{Open sets with constant resolvent norm}
\label{sec:const}

We consider a closed operator $T$ that acts in a complex uniformly convex Banach space $\X$. 
We recall here the definition, \cf~\cite{Globevnik-1975-47}. 
\begin{define}\label{def.cuc}
 A complex normed space $\X$ is called \emph{complex uniformly convex} if for every $\varepsilon>0$ there exists $\delta>0$ such that 
for all $x,y\in\X${\rm:}
\begin{equation}
\forall\,\zeta\in B_1(0):\,\|x+ \zeta y\|\leq 1, \quad \|y\|>\varepsilon\quad  \Longrightarrow \quad \|x\|< 1-\delta. 
\end{equation}
\end{define}

\begin{theorem}\label{thm.const.norm}
Let $T$ be a closed operator in a complex uniformly convex Banach space $\X$.
If there exist an open subset  $U\subset\rho(T)$ and a constant $M>0$ such that
\begin{equation}
\|(T-\lambda)^{-1}\| =M, \quad \lambda \in U, 
\end{equation}
then $\|(T-\lambda)^{-1}\|\geq M$ for all $\lm\in\rho(T)$.
\end{theorem}
\begin{proof}
The proof is based on \cite[Lem.~1.1]{Globevnik-1974-15} and a straightforward generalisation of \cite[Lem.~3.0]{Globevnik-1974-15} to the complex uniformly convex spaces.

Without loss of generality we assume that $M=1$.
Let $\lambda_0 \in U$. Then the resolvent can be expanded as
\begin{equation}
 f(\zeta):=(T-(\lm_0+\zeta))^{-1}=\sum_{j=0}^{\infty}A_j\,\zeta^j, \quad A_0:=(T-\lm_0)^{-1}, \quad A_j:=A_0^{j+1},
\end{equation}
for $\zeta \in \C$ with $|\zeta|<1/\|A_0\|=1$. 
So there exists a neighbourhood of $\zeta=0$ where $f$ is analytic and $\|f\|\equiv 1 =\|A_0\|$. Hence we can apply \cite[Lem.~1.1]{Globevnik-1974-15} which
yields that for every index $j>0$ there exists $r_j>0$ such that 
$\|A_0+\zeta A_j\|\leq \|A_0\|=1$, $|\zeta|\leq r_j$.
This implies that every $u\in \X$ with $\|u\|=1$ satisfies
\begin{equation}\label{estimate.for.x}
\forall\,\zeta\in B_1(0):\quad  \|A_0u+\zeta\, r_j A_j u \|\leq 1.
\end{equation}

There exists a sequence $\{e_k\}_k\subset\X$ with $\|e_k\|=1$ such that 
\begin{equation}\label{lim.A0}
\lim_{k\to\infty} \|(T-\lm_0)^{-1}e_k\|=\lim_{k\to\infty}\| A_0 e_k \|=\|A_0\|= 1.
 \end{equation}
Define $x_k:= A_0 e_k$. Then~\eqref{lim.A0} can be rewritten as $\|x_k\| \to 1$.
Assume that there exist $\varepsilon>0$ and an infinite subset $I\subset\N$ such that $y_k:=r_1  A_1 e_k$ satisfies $\|y_k\|>\varepsilon$, $k\in I$.
The inequality~\eqref{estimate.for.x} with $j:=1$ and $u:=e_k$ implies
\begin{equation}
\forall\,\zeta\in B_1(0):\quad  \|x_k+\zeta y_k\|\leq 1.
\end{equation}
The complex uniform convexity of $\X$ yields the existence of some $\delta>0$ such that $\|x_k\|< 1-\delta$, $k\in I$; this is a contradiction to $\|x_k\| \to 1$.
Therefore  $\|y_k\| \to 0$, and hence 
\begin{equation}\label{lim.A1}
 \lim_{k\to\infty} \|(T-\lm_0)^{-2}e_k\|=\lim_{k\to\infty} \|A_1e_k\|=0.
\end{equation}

Now, for an arbitrary $\lm\in\rho(T)$, using twice the first resolvent identity, we obtain
\begin{equation}
\begin{aligned}
(T-\lm)^{-1}-(T-\lm_0)^{-1} & = (\lm-\lm_0) (T-\lm)^{-1}(T-\lm_0)^{-1}
\\ 
& = (\lm-\lm_0) 
\left(
I + (\lm-\lm_0) (T-\lm)^{-1}
\right) (T-\lm_0)^{-2}.
\end{aligned}
\end{equation}
So we have
\begin{equation}\label{eq.res.difference.const}
\begin{aligned}
\left\|(T-\lm)^{-1} \right\| & \geq \|(T-\lm_0)^{-1}e_k\| 
\\
& \quad -   
|\lm-\lm_0| \|I+(\lm-\lm_0) (T-\lm)^{-1} \| \|(T-\lm_0)^{-2}e_k \|.
\end{aligned}
\end{equation}
Finally, the limits \eqref{lim.A0} and \eqref{lim.A1} yield $\|(T-\lm)^{-1}\|\geq 1 =M$.
\end{proof}

Theorem~\ref{thm.const.norm} yields another proof for the following result, \cf~\cite{Shargorodsky-2008-40,Boettcher-1997-3,Shargorodsky-2010-42}.
\begin{corollary}\label{cor.bdd.or.semigroup}
Let $\X$ be a complex uniformly convex Banach space. If a closed operator $T$ is bounded or generates a $C_0$ semigroup, then its resolvent norm cannot be constant on any open subset of $\rho(T)$.
\end{corollary}

\begin{proof}
The claim immediately follows from Theorem~\ref{thm.const.norm} since
the resolvent norm decays along an infinite ray in both cases;
 for $T$ bounded we have $\|(T-\lm)^{-1}\|\leq 1/(|\lm|-\|T\|)$ for all $\lm\in\C$ with $|\lm|>\|T\|$ 
and, for $T$ generating a $C_0$ semigroup, the Hille-Yosida Theorem, \cf~for instance \cite[Thm.~II.3.8]{Engel-Nagel}, yields the existence of $C>0$, $\omega\in\R$ such that $\|(T-\lm)^{-1}\|\leq C/(\lm-\omega)$ for all real $\lm>\omega$. 
\end{proof}

Theorem~\ref{thm.const.norm} can be generalised for higher powers of the resolvent. 
In an application to $n$-pseudospectra we set $l:=2^n$.
\begin{theorem}\label{thm.const.norm.n}
Let $l\in\N$ and let $T$ be a closed operator in a complex uniformly convex Banach space $\X$.
If there exist an open subset  $U\subset\rho(T)$ and a constant $M>0$ such that
\begin{equation}
\|(T-\lambda)^{-l}\| =M, \quad \lambda \in U, 
\end{equation}
then $\|(T-\lambda)^{-l}\|\geq M$ for all $\lm\in\rho(T)$.
\end{theorem}

\begin{proof}
Without loss of generality we assume that $M=1$.
 For some $\lm_0\in U$ we expand
\begin{equation}
 f(\zeta):=(T-(\lm_0+\zeta))^{-l}=(T-\lm_0)^{-l}+l(T-\lm_0)^{-(l+1)}\zeta+\mathcal O(\zeta^2),
\end{equation}
for $\zeta\in\C$ with $|\zeta|< 1/\|(T-\lm_0)^{-1}\|$.
By proceeding analogously as in the proof of Theorem~\ref{thm.const.norm},
we find a normalised sequence $\{e_k\}_k\subset\X$ such that
\begin{equation}\label{eq.re.differnece.const.l}
 \lim_{k\to\infty}\left\|(T-\lm_0)^{-l}e_k\right\|=1,\quad \lim_{k\to\infty}\left\|(T-\lm_0)^{-(l+1)}e_k\right\|=0.
\end{equation}
Take an arbitrary $\lm\in\rho(T)$. By the first resolvent identity and binomial theorem, we obtain
\begin{equation}
\begin{aligned}
&(T-\lm)^{-l}-(T-\lm_0)^{-l}= B_{\lm}\, (T-\lm_0)^{-(l+1)},\\
&B_{\lm}:=\sum_{j=0}^{l-1}\begin{pmatrix}l\\ j+1\end{pmatrix} (\lm-\lm_0)^{j+1}\left(I+(\lm-\lm_0)(T-\lm)^{-1}\right)^{j+1}(T-\lm_0)^{-j}.
\end{aligned}
\end{equation}
Note that $B_{\lm}$ is a bounded operator.
Now, in a way analogous to~\eqref{eq.res.difference.const}, one may show that~\eqref{eq.re.differnece.const.l}  implies $\|(T-\lm)^{-l}\|\geq 1=M$.
\end{proof}

\begin{remark}
If a closed operator $T$ in a complex uniformly convex Banach space has decaying resolvent norm $\|(T-\lm)^{-1}\|$ as $\lm\in\rho(T)$ tends to infinity along some path, 
then also each power of the resolvent has decaying norm.
In this case, by Theorem~\ref{thm.const.norm.n}, no power of the resolvent can have constant norm on any open subset of $\rho(T)$.
Therefore assumption~\ref{ass.const.n} of Theorem~\ref{thm.psp.conv.n} is satisfied for any $n\in\N_0$.
This applies in particular for operators that are bounded or generators of $C_0$ semigroups, \cf~Corollary~\ref{cor.bdd.or.semigroup} and its proof.
\end{remark}

We present several examples illustrating various behaviours of the resolvent norm. All belong to a class inspired by \cite{Balakrishnan-1993-6} and have the form of block operator matrices acting in a Hilbert space $\H \oplus \H$:
\begin{equation}\label{form.of.matrixA}
\mathcal{A}: = 
\begin{pmatrix}
0 & f(A) \\
A & 0
\end{pmatrix}, \qquad \Dom(\mathcal{A}) := \Dom(A) \oplus \Dom(f(A)),
\end{equation}
where $A = A^*$ is a strictly positive operator
and $f :\R \to \R$ is a continuous positive function such that $\lim_{x\to\infty}f(x)=C$ for some $C\in [0,\infty]$. 
It is easy to verify that $\mathcal{A}$ is a closed operator. 
Moreover, it follows from~\cite[Thm.~2.3.7~i)]{Tretter-2008} that $\lm\in\rho(\mathcal A)$ if and only if $0\in\rho(f(A)-\lm^2A^{-1})$.
If $C\neq 0$, one may verify that the latter is equivalent to
$\lambda^2 \in\rho(Af(A))$; then
\begin{equation}\label{A.res.form}
(\mathcal A -\lambda)^{-1} = 
\begin{pmatrix}
\lambda (Af(A)-\lambda^2)^{-1} & f(A)(Af(A)-\lambda^2)^{-1} \\
A (Af(A)-\lambda^2)^{-1} & \lambda (Af(A)-\lambda^2)^{-1}
\end{pmatrix}.
\end{equation}
Various further assumptions on the spectrum of $A$ and the function $f$ are imposed in the individual examples below.

 If the resolvent of $A$ is assumed to be compact, we denote by $\{\alpha_k\}_k$, $\{e_k\}_k$ the sets of eigenvalues and corresponding eigenvectors of $A$.
Inspired by the strategy in \cite[Thm.3.2]{Shargorodsky-2008-40}, we use that $\span\{(e_k,0)^t,(0,e_k)^t\},$ $k\in \N$, are orthogonal invariant subspaces of $\mathcal A$ and their span is dense in $\H \oplus \H$.
The matrix representation of $\mathcal A$ with respect to $\{(e_k,0)^t,(0,e_k)^t\}$ is 
\begin{equation}\label{eq.def.B.k}
B_k:=
\begin{pmatrix}
0 & f(\alpha_k) \\
\alpha_k & 0
\end{pmatrix}.
\end{equation}

If we take $\H:=L^2(\R^d)$, $A := \pm \Delta$, $\Dom(A) := W^{2,2}(\R^d)$ and $f(x):=1$, then $\mathcal{A}$ corresponds to a generator of the Klein-Gordon or wave equation (without potentials). 
However, the crucial point is that we consider a ``wrong space'' for $\mathcal{A}$, \ie~not the energy space, therefore $\mathcal A$ is non-selfadjoint in both latter examples. 

In spite of the simple structure of $\mathcal{A}$, the variety of resolvent behaviours appears to be quite rich:
\begin{enumerate}[label=\rm{\roman{*})}]
 \item \label{ex.ass.empty} If $A$  has  compact resolvent and $f(x) \to 0$ as $x \to \infty$, then $\rho(\mathcal A) =\emptyset$, \cf~Example \ref{ex.empty.res}.
 \item \label{ex.ass.const} Assume that $A$  has  compact resolvent and there exists a constant $C>0$ such that $f(x) \to C$ as $x \to  \infty$. If, in addition,
\begin{equation}\label{assumption.example.const.norm}
\exists\,m\geq 0, \,\forall\,k\in\N:\quad f(\alpha_k)^2 \geq C^2 -\frac{m}{\alpha_k},
\end{equation}
then $\mathcal A$ has constant resolvent norm on a non-empty open set, \cf~Example~\ref{ex.const.norm}.
 
Notice that Shargorodsky's example, \cf~\cite[Thm.~3.2]{Shargorodsky-2008-40}, can be written in the form~\eqref{form.of.matrixA}. To this end, set $A:= \diag (\alpha_1,\alpha_2,\alpha_3, \dots)$, $\alpha_k\geq 2$, $\alpha_k\to \infty$, $f(x) := 1 + 1/x$ and perform the unitary transform $(e_k,0)^t \mapsto (0,e_k)^t$ and $(0,e_k)^t \mapsto (e_k,0)^t$. 
Assumption~\eqref{assumption.example.const.norm} is satisfied with $C:=1$ and $m:=0$.
\item \label{ex.ass.non.const} If $A$ and $f$ are as in ii), however, the condition \eqref{assumption.example.const.norm} is not satisfied, then the resolvent norm is not constant on any open set, \cf~Example~\ref{ex.non.const.norm}.
\item \label{ex.ass.sg} Finally, if $f(x) = |x|^\beta$ with some $\beta \in (0,1)$, then, for $\lm:=r\,{\rm e}^{\ii\varphi}$ with $\varphi\notin~\{0,\pi\}$, the resolvent norm decays in the limit $r\to\infty$, with
\begin{equation}\label{res.norm.decay.estimate}
\|(\mathcal A-r {\rm e}^{\ii\varphi})^{-1}\| = \mathcal O( r^{-2\beta / (1+\beta)} ).
 \end{equation}
 If, in addition, $[c,\infty) \subset \sigma(A)$ for some $c>0$, then for every $\varphi \notin \{0,\pi\}$ there exists $\omega>0$ such that 
\begin{equation}\label{ex.sg.dr}
r^{2\beta / (1+\beta)} \|(\mathcal A-r {\rm e}^{\ii\varphi})^{-1}\| \tolong \omega, \quad  r\to  \infty,
\end{equation}
\cf~ Example~\ref{ex.not.semi}. 
Note that $2\beta/(1+\beta)<1$.
Hence, under the additional assumption  $[c,\infty) \subset \sigma(A)$, the constructed operator matrix $\mathcal A$ has non-compact resolvent and does not generate any $C_0$ semigroup, \cf~the Hille-Yosida Theorem \eg~in~\cite[Thm.~II.3.8]{Engel-Nagel}; 
nonetheless, the resolvent norm decay and Theorem~\ref{thm.const.norm} exclude  the occurrence of constant resolvent norm on any open set.
\end{enumerate}

\begin{example}[Empty resolvent set]\label{ex.empty.res}
Let $A$ and $f$ satisfy assumption \ref{ex.ass.empty}. Consider unit vectors $(u_k e_k, v_k e_k)^t$, \ie~$u_k, v_k \in \C$ with $|u_k|^2 + |v_k|^2=1$, and observe that, for any $\lambda \in \C$,
\begin{equation}\label{A.action.1}
(\mathcal A - \lambda) 
\begin{pmatrix}
u_k e_k \\
v_k e_k
\end{pmatrix}
=
\begin{pmatrix}
(-\lambda u_k +  f(\alpha_k) v_k )e_k \\
( \alpha_k u_k -\lambda  v_k )e_k
\end{pmatrix}.
\end{equation}
If we set  $u_k = \lambda v_k / \alpha_k$, the norm of the r.h.s. of \eqref{A.action.1} is
$
|\alpha_k \, f(\alpha_k) - \lambda^2 |/ \sqrt{\alpha_k^2 + |\lambda|^2}$
and the latter tends to zero as $k \to \infty$.
This implies that $\lm\in\sigma_{\rm app}(T)\subset\sigma(T)$.
\end{example}

\begin{example}[Constant resolvent norm on an open set]\label{ex.const.norm}
Let $A$ and $f$ satisfy assumption \ref{ex.ass.const}. 
Then, with $B_k$ as defined in~\eqref{eq.def.B.k},
\begin{equation}\label{eq.norm.is.sup}
\|(\mathcal A -\lambda)^{-1} \| 
=\sup_k \|(B_k - \lambda)^{-1}\| .
\end{equation}
For $\lambda\in\rho(\mathcal A)$ we have 
%
%
\begin{equation}\label{conv.res.norm.of.B.k}
\begin{aligned}
\lim_{k\to\infty}\|(B_k - \lambda)^{-1}\|
&= \lim_{k\to\infty}\left\|
(\alpha_k f(\alpha_k)-\lambda^2)^{-1}
\begin{pmatrix}
\lambda  & f(\alpha_k) \\
\alpha_k  & \lambda 
\end{pmatrix}
\right\|\\
&=\left\|
\begin{pmatrix}
 0 & 0\\ \frac{1}{C} & 0
\end{pmatrix}
\right\|=\frac{1}{C}.
\end{aligned}
\end{equation}
It is shown below that assumption~ \eqref{assumption.example.const.norm} yields the existence of an open subset $U\subset\C$ such that $\|(B_k - \lambda)^{-1}\| < 1/C$ for all $\lambda\in U$ and $k\in\N$. 
Therefore $\|(\mathcal A -\lambda)^{-1} \| = 1/C$ on $U$.

We write $\lambda = r {\rm e}^{\ii \varphi}$. Simple manipulations reveal that
\begin{equation}\label{Mn.est.1}
\|(B_k-\lambda)^{-1}\|^2 \leq 
\frac{r^2 + 2 r \max\{\alpha_k,f(\alpha_k)\} + \max\{\alpha_k,f(\alpha_k)\}^2}
{r^4- 2 \alpha_k f(\alpha_k) r^2 \cos 2 \varphi+\alpha_k^2 f(\alpha_k)^2}. 
\end{equation}
We select $k_1 \in \N$ such that , for all $k\geq k_1$, we have $f(\alpha_k) \leq \alpha_k$ and therefore $\max\{\alpha_k,f(\alpha_k)\}=\alpha_k$.
We can find $r_1>0$ such that the r.h.s. of 
\eqref{Mn.est.1} is strictly less than $1/C^2$ for all $r\geq r_1$ and $k < k_1$. 
Then, for $r\geq r_1$, the inequality $\|(B_k-\lambda)^{-1}\| < 1/C$ is satisfied for all $k\in\N$ if
\begin{equation}\label{ineq.const.norm.2}
\frac{r^2(r^2 - C^2)}{\alpha_k^2} 
- \frac{2r(C^2 + f(\alpha_k) r \cos 2 \varphi )}{\alpha_k} + f(\alpha_k)^2 -C^2 >0, \quad k\geq k_1.
\end{equation}
The assumption~\eqref{assumption.example.const.norm} guarantees that \eqref{ineq.const.norm.2} holds if $\varphi\in [0,2\pi)$ is chosen such that $\cos 2 \varphi <0$ and $r$ is sufficiently large, namely $r\geq r_0$ for some fixed $r_0\geq\max\{r_1,C\}$ that satisfies 
\begin{equation}\label{eq.cond.for.r0}
2 r_0\left(f(\alpha_k)r_0|\cos 2\varphi|- C^2\right)> m, \quad k\geq k_1.
\end{equation}

In Shargorodsky's example the resolvent norm of the considered operator $\mathcal A$ is constantly equal to $C:=1$ on $U:=B_{1/2}(0)$, \cf~\cite[Thm.~3.2]{Shargorodsky-2008-40}. 
By the above reasoning, this is also true for all $\lambda=r {\rm e}^{\ii\varphi}\in\C$ 
such that $\cos 2\varphi<0$ and $r\geq r_0$ for some $r_0\geq 1$ that satisfies~\eqref{eq.cond.for.r0} with $k_1:=1$, $f(x):=1+1/x$ and $m:=0$.
One may check that this is satisfied for $r_0=1/|\cos 2\varphi|$.
\end{example}

\begin{example}[Non-constant resolvent norm on any open set]\label{ex.non.const.norm}

Let $A$ and $f$ satisfy assumption \ref{ex.ass.non.const}. 
We show below that $\|(\mathcal A-\lm)^{-1}\|>1/C$ for every $\lm\in\rho(\mathcal A)$.
Then, for any fixed $\lambda_0\in\rho(\mathcal A)$ and $\delta_0:=\|(\mathcal A-\lm_0)^{-1}\|-1/C>0$, there exists an open bounded neighbourhood $V_0$ of $\lm_0$ such that $\overline{V_0}\subset\rho(\mathcal A)$ and 
\begin{equation}
 \|(\mathcal A-\lm)^{-1}\|>\frac{1}{C}+\frac{\delta_0}{2}, \quad \lm\in V_0.
\end{equation}
For any $\lm\in \overline{V_0}$, by \eqref{conv.res.norm.of.B.k}, we have $\|(B_k-\lm)^{-1}\|\to 1/C$ as $k\to\infty$. Since $\overline{V_0}$ is compact, the convergence is uniform on $\overline{V_0}$ and hence there exists $k_0\in\N$
such that 
\begin{equation}
 \|(B_k-\lm)^{-1}\|\leq \frac{1}{C}+\frac{\delta_0}{2}, \quad \lm\in \overline{V_0}, \quad k> k_0.
\end{equation}
Define $\mathcal A_{k_0}:=\diag(B_1,\dots,B_{k_0})$. Then we have
$\|(\mathcal A-\lm)^{-1}\|=\|(\mathcal A_{k_0}-\lm)^{-1}\|$ for all $\lm\in V_0$.
Since $\mathcal A_{k_0}$ acts in a finite-dimensional space, it is a bounded operator and so its resolvent norm cannot be constant anywhere, hence $\lm\mapsto\|(\mathcal A-\lm)^{-1}\|$ is nowhere constant in $V_0$.
Now because $\lm_0\in\rho(\mathcal A)$ was arbitrary, the same holds for the whole resolvent set.

It is left to show that $\|(\mathcal A-\lm)^{-1}\|>1/C$ for every $\lm\in\rho(\mathcal A)$.
Let $\lm=r {\rm e}^{\ii\varphi}\in\rho(\mathcal A)$. It suffices to show the existence of a $k_{\lm}\in\N$ such that
$\|(B_{k_{\lambda}}-\lambda)^{-1}\|> 1/C$.
With the use of 
 \begin{equation}
 \|(B_k - \lambda)^{-1}\|
 \geq \|(B_k - \lambda)^{-1}(e_k,0)^t\|\\
 =\frac{\sqrt{|\lm|^2+\alpha_k^2} }{|\lambda^2 - \alpha_k f(\alpha_k)|}, \quad k\in\N,
  \end{equation}
we see that, for every $k\in\N$,
\begin{equation}
\|(B_{k}-\lambda)^{-1}\|^2
\geq 
\frac{r^2 +\alpha_k^2}
{r^4- 2 \alpha_k f(\alpha_k) r^2 \cos 2 \varphi+\alpha_k^2 f(\alpha_k)^2}. 
\end{equation}
The r.h.s. of the latter is strictly larger than $1/C^2$ if and only if
\begin{equation}\label{ineq.const.norm.3}
\frac{r^2(r^2 - C^2)}{\alpha_k^2} 
- \frac{2 f(\alpha_k) r^2 \cos 2 \varphi}{\alpha_k} + f(\alpha_k)^2 -C^2 <0.
\end{equation}
Since $\alpha_k\to\infty$ and $f(\alpha_k)\to C$ as $k\to\infty$, there exists $m_{\lambda}>0$ such that
\begin{equation}
\frac{r^2(r^2 - C^2)}{\alpha_k} 
- 2 f(\alpha_k) r^2 \cos 2 \varphi \leq m_{\lambda}, \quad k\in\N.
\end{equation}
If assumption~\eqref{assumption.example.const.norm} is not satisfied, 
then there exists $k_{\lambda}\in\N$ such that \eqref{ineq.const.norm.3} is satisfied for $k:=k_{\lambda}$, and hence the claimed estimate $\|(B_{k_{\lambda}}-\lambda)^{-1}\|> 1/C$ holds.
\end{example}

\begin{example}[Decaying resolvent norm and lack of semigroup generation]\label{ex.not.semi}
Let $A$ and $f$ satisfy assumption \ref{ex.ass.sg}. 
The resolvent formula~\eqref{A.res.form} yields that $\sigma(\mathcal A) = \{\pm \mu^{(1+\beta)/2}\, : \, \mu \in \sigma(A)\} \subset \R$. 
Let $\lambda = r {\rm e}^{\ii \varphi}$ with some fixed $\varphi \notin \{0,\pi\}$.

In the limit $|\lm|\to\infty$, the dominant term in $\|(\mathcal A - \lambda)^{-1}\|$ corresponds to the down-left entry, \cf~\eqref{A.res.form}. The latter can be verified using 
\begin{equation} 
 \begin{aligned}
 \|A^\beta(A^{1+\beta} - \lm^2)^{-1}\| &\leq \|A^{\beta-1}\| \|A(A^{1+\beta} - \lm^2)^{-1}\|, \\
\| \lm (A^{1+\beta} - \lm^2)^{-1}\| &\leq |\lm| / \dist( \lm^2, [0,\infty))=\mathcal O (|\lm|^{-1}).
\end{aligned}
\end{equation}
Since 
\begin{equation}\label{eq.norm.of.down-left}
\|A(A^{1+\beta}-\lambda^2)^{-1}\| = \sup 
\left\{ \frac{\mu}{|\mu^{1+\beta} - \lambda^2| } \, : \: \mu \in \sigma(A)
\right\},
\end{equation}
we analyse the behaviour of the following function (its supremum over $\mu\in\sigma(A)$ is the square of the norm in~\eqref{eq.norm.of.down-left}):
\begin{equation}
g(\mu) := \frac{\mu^2}{r^4  - 2 \mu^{1+\beta} r^2 \cos 2 \varphi + \mu^{2(1+\beta)}}, \quad \mu >0.
\end{equation}
For every $r>0$ the maximum of $g$ is attained at some point $\mu_0(r)>0$. Elementary calculations show that  $\mu_0(r) \sim {\rm const}\cdot r^{2/(1+\beta)}$, hence the estimate~\eqref{res.norm.decay.estimate} on the decay of the resolvent follows. 

If, in addition, $[c,\infty) \subset \sigma(A)$ for some $c>0$, then $\mu_0(r) \in \sigma(A)$ for all sufficiently large $r$, and we obtain \eqref{ex.sg.dr}.
\end{example}

\section{Proof of pseudospectral convergence}
\label{sec:proof.1}

We divide the proof  of Theorem \ref{thm.psp.conv} in several lemmas. First we recall that 
the region of boundedness of a sequence of closed operators $T_k$ acting in $\H_k$ is the set
\begin{equation}
\begin{aligned}
\Delta_{\rm b}(\{T_k\}_k )
:=\Big\{
\lambda \in\C:\,\exists\, k_\lm \in \N, \,  \exists\, M_\lm>0, \,  \forall \, k \geq k_\lm: \, 
\\
\lambda\in\rho(T_k) \mbox{ and } \|(T_k-\lambda)^{-1}\| \leq M_\lm
\Big\},
\end{aligned}
\end{equation}
\cf~\cite[Sec.~VIII.1.1]{Kato-1966}.
To simplify the notation, in the sequel we denote the resolvents by 
$R(\lambda):=(T-\lambda)^{-1}$ and $R_k(\lambda):=(T_k-\lambda)^{-1}$.

The following result is a generalisation of standard results; claim i) is a generalisation of~\cite[Lem.~2.1]{Harrabi-1998} where bounded operators are considered that converge in norm,
claim ii) was shown in \cf~\cite[Thm.~IV.2.25, Sec.~IV.3.3]{Kato-1966} for the case of the usual norm resolvent convergence.
\begin{lemma}\label{lem.rb}
Let $T, T_k$ be densely defined. If $T_k \gnrlong T$, then 
\begin{enumerate}[label=\rm{\roman{*})}]
\item the region of boundedness is $\Delta_{\rm b}\left(\{T_k\}_k \right) = \rho(T)$;
\item for all $\lm\in\rho(T)$,
\begin{equation}\label{eq.gnr.forall.lm}
\|R_k(\lm) P_k - R(\lm)P \|\tolong~0, \quad k\to\infty.
\end{equation}
\end{enumerate}
\end{lemma}
\begin{proof}
We proceed in three steps:
\begin{enumerate}[label=\rm{\alph{*})}]
\item If there exists $\lambda_0 \in \cap_{k\in\N}\rho(T_k)\cap\rho(T)$ such that $R_k(\lm_0)P_k \s R(\lm_0)P$, then $\sigma_{\rm app}(T)\subset \C\backslash\Delta_{\rm b} (\{T_k\}_k )$.
\item If there exists $\lambda_0 \in \cap_{k\in\N} \rho(T_k)\cap\rho(T)$ such that $R_k(\lm_0) P_k \s R(\lm_0)P$ and $R_k(\lm_0)^* P_k \s R(\lm_0)^* P$, then $\sigma(T)\subset \C\backslash \Delta_{\rm b} (\{T_k\}_k )$.
\item If there exists $\lambda_0 \in \cap_{k\in\N} \rho(T_k)\cap\rho(T)$ such that $\|R_k(\lm_0) P_k - R(\lm_0)P \|\to~0$, then $\rho(T) \subset \Delta_{\rm b} (\{T_k\}_k )$ and~\eqref{eq.gnr.forall.lm} holds for all $\lm\in\rho(T)$.
\end{enumerate}
Claim b) implies $\Delta_{\rm b} (\{T_k\}_k ) \subset \rho(T)$; equality then follows from claim c). Note that the generalised strong resolvent convergence of $T_k$ and $T_k^*$ to the respective limit is given by the generalised norm resolvent convergence.

Claim a):
Let $\mu\in\sigma_{\rm app}(T)$.
If there exists an infinite set $I\subset\N$ such that $\mu\in\sigma(T_k)$ for all $k\in I$, then obviously $\mu \notin \Delta_{\rm b} (\{T_k\}_k )$.
In the other case there exists $k_0\in\N$ such that $\mu\in\rho(T_k)$ for all $k\geq k_0$.
Since $\mu\in\sigma_{\rm app}(T)$, there exists $\{x_m\}_m \subset \Dom(T)$ such that $\|x_m\|=1$, $\|(T-\mu)x_m\|\to 0$.
Define 
\begin{equation}
x_{m;k}:=R_k(\lm_0) P_k (T-\lm_0) x_m, \quad m\in\N, \,k \geq k_0.
\end{equation}
Since,
for all $y \in \H$, 
\begin{equation}
(P_k-P) R(\lm_0) y = (I-P_k)( R_k(\lm_0)P_k - R(\lm_0)P )y \tolong 0, \quad k\to\infty,
\end{equation}
we have $P_k x \to Px=x$ for all $x \in \Dom(T)$.
By the density of $\Dom(T)\subset\H$ and $\|P_k\| = 1$, the same is true for all $x\in\H$.
Then the assumptions imply $x_{m;k}\to x_m$ and $T_kx_{m;k}\to Tx_m$ as $k\to\infty$. 
Hence there exists a strictly increasing sequence $\{k_m\}_m\subset \N$ such that, for every $m\in\N$, the element $y_m:=x_{m;k_m}$ satisfies 
\begin{equation}
\|y_m-x_m\|+\|T_{k_m}y_{m}-Tx_m\|<\frac{1}{m}.
 \end{equation}
Therefore
\begin{equation}
 \begin{aligned}
 \|R_{k_m}(\mu)\|
 \geq 
 \frac{\|y_{m}\|}{\|(T_{k_m}-\mu)y_{m}\|}
 \geq 
 \frac{1-\frac{1}{m}}{\|(T-\mu)x_{m}\| + (1+|\mu|)\frac{1}{m} } \tolong \infty,
 \end{aligned}
\end{equation}
thus $\mu\notin\Delta_b\left(\{T_k\}_k\right)$.

Claim b): 
We split $\sigma(T)=\sigma_{\rm app}(T)\cup\sigma_{\rm app}(T^*)^*,$ where, for a set $\Omega \subset \C$, we denote $\Omega^*:=\{\overline{z}\, : \, z \in \Omega\}$. The claim follows from a) applied to $T,T_k$ or $T^*,T_k^*$; note that $\Delta_{\rm b} (\{T_k\}_k ) =\Delta_{\rm b} (\{T_k^*\}_k )^*$.

Claim c):
Since $\|R_k(\lm_0) P_k - R(\lm_0) P\|\to 0$, we have $\lambda _0\in \Delta_{\rm b} (\{T_k\}_k ) \cap \rho(T)$.
Let $\lambda \in \rho(T) \setminus \{\lm_0\}$.
The~spectral mapping theorem yields 
$(\lambda-\lambda_0)^{-1} \in \rho \left(R(\lm_0)\right)$. Since the latter set may only differ by the element $0$ to $\rho \left(R(\lm_0) P \right)$, we have $(\lambda-\lambda_0)^{-1} \in\rho\left(R(\lm_0)P\right)$.
By Kato's result~ \cite[Theorem~IV.2.25]{Kato-1966}, there exists $k_0 \in \N$ such that 
\begin{equation}
(\lambda-\lambda_0)^{-1}\in \rho \left(R_k(\lm_0) P_k \right) \subset \rho\left( R_k(\lm_0) \right), \quad k \geq k_0.
\end{equation}
Again by the spectral mapping theorem, we obtain $\lambda \in \rho(T_k), \, k\geq k_0$.
In order that $\lambda \in \Delta_{\rm b} (\{T_k\}_k )$, it is left to show that $\|R_k(\lm)\|$, $k\geq k_0$, are uniformly bounded.
The idea is to show~\eqref{eq.gnr.forall.lm}, then, in particular, the resolvents are uniformly bounded.

A straightforward application of the first resolvent identity yields
\begin{equation}\label{eq.res.diff.for.lm}
\left(R_k(\lm)P_k-R(\lm)P\right)S_k=\left(I+(\lm-\lm_0)R(\lm)P\right)\left(R_k(\lm_0)P_k-R(\lm_0)P\right),
\end{equation}
where
\begin{equation}
S_k:=I+(\lm_0-\lm)R_k(\lm_0)P_k.
\end{equation}
Since $R_k(\lm_0)P_k \to R(\lm_0)P$ and $S:=\lim_{k \to \infty} S_k$ has bounded inverse, the operator $S_k$ is boundedly invertible for all sufficiently large $k$,
and $\|S_k^{-1}\|$ is uniformly bounded, \cf~\cite[Sec.~I.4.4, Thm.~IV.1.16]{Kato-1966}. Now~\eqref{eq.gnr.forall.lm} follows from \eqref{eq.res.diff.for.lm} and $R_k(\lm_0)P_k \to R(\lm_0)P$. 
\end{proof}

\begin{lemma}\label{lem.K}
Let $T, T_k$ be densely defined.
Assume that $T_k \gnrlong T$. 
Let $K \subset \sigma_\varepsilon(T)$ be compact. Then there exists $k_0 \in \N$ such that $K \subset \sigma_\varepsilon(T_k)$  for all $k\geq k_0$.
\end{lemma}

\begin{proof}
First we show that
for every $\lambda \in \sigma_\varepsilon(T)$ there exist $ r_\lambda >0$ and $k_\lambda  \in \N$ such that $B_{r_\lambda}(\lambda) \subset \sigma_{\varepsilon}(T_k)$ for all $k\geq k_\lambda$. Then the claim follows by the compactness of $K\subset \sigma_\varepsilon(T)$.
We divide the proof into two cases: i) $\lambda \in \sigma(T)$, ii) $\lambda \in \rho(T) \cap \sigma_\varepsilon(T)$.

Case i): We proceed by contradiction. 
Assume that there exists $\lambda \in \sigma(T)$ such that for all $k' \in \N$ and $r>0$ there exist  $k \geq k'$ and $\lm_k \in B_r(\lambda)$  with $\lm_k \notin \sigma_\varepsilon(T_k)$, \ie~$\|R_k(\lm_k)\|\leq 1/\varepsilon$.
If we relate $k'$ and $r$ by $r_{k'}:=1/{k'}$, we obtain an infinite set $I \subset \N$ and elements $\lm_k\in\rho(T_k)$, $k\in I$, such that 
\begin{equation}\label{Rn.ineq.1}
\|R_k(\lm_k)\| \leq \frac{1}{\varepsilon}, \quad k \in I, \quad\text{and}\quad \lm_k\tolong\lm, \quad k\to\infty. 
\end{equation}
Since $\lm \notin \rho(T)$, Lemma~\ref{lem.rb}~i) implies $\lm\notin\Delta_{\rm b}(\{T_k\}_k)$. 
Hence there exist infinite sets $I_1 \subset I$, $I_2 \subset I$ such that either $\lm \in \sigma(T_k)$, $k \in I_1$, or $\|R_k(\lm)\| \to \infty$, $k\in I_2$. 
The first option contradicts the first property in \eqref{Rn.ineq.1} since, for $k\in I_1$,
\begin{equation}\label{Rn.ineq.2}
\|R_k(\lm_k) \| \geq \frac{1}{\dist(\lm_k,\sigma(T_k))} \geq \frac{1}{|\lm-\lm_k|}\tolong \infty, \quad k\to\infty.  
\end{equation}
The second option also contradicts~\eqref{Rn.ineq.1} since, for $k\in I_2$,
\begin{equation}
\|R_k(\lm_k)\| \geq \|R_k(\lm)\| - \|R_k(\lm) - R_k(\lm_k)\|
\end{equation}
and the second term tends to zero as $k \to \infty$. 
In order to justify the latter, we expand the resolvents around the points~$\lm_k$ and estimate
\begin{equation}\label{Rn.ineq.3}
\|R_k(\lm) - R_k(\lm_k)\| 
\leq 
\frac{|\lm -\lm_k| \|R_k(\lm_k)\|^2}{1 - |\lm -\lm_k| \|R_k(\lm_k)\|}
\leq \frac{|\lm -\lm_k| \frac{1}{\varepsilon^2}}{1 - |\lm -\lm_k| \frac{1}{\varepsilon}}
\tolong 0, \quad k\to\infty.
\end{equation}

Case ii): 
As $\lm \in \sigma_\varepsilon(T) \cap \rho(T)$, we can write $\|R(\lm)\|= 1/\varepsilon + \alpha$ with some $\alpha >0$.
Since $\lm\in\rho(T)$,  Lemma~\ref{lem.rb}~i) implies that there exist $\widetilde k_\lambda \in \N$ and $M_{\lambda}>0$ such that $\lambda \in \rho(T_k)$ and $\|R_k(\lm)\|\leq M_{\lambda}$ for all $k\geq\widetilde k_\lm$. 
The generalised norm resolvent convergence $T_k \gnrlong T$  and Lemma~\ref{lem.rb}~ii) yield
the existence of $k_{\lambda}\geq \widetilde k_{\lm}$ such that $\|R_k(\lm)P_k - R(\lm)P\|< \alpha/2$ for all $k\geq k_{\lm}$.
The resolvent expansion around $\lm$ implies that, for every positive $r_{\lm}<1/M_{\lm}$ and $k\geq k_\lm$, we have $B_{r_{\lambda}}(\lambda) \subset \rho(T_k)$, and, for all $\mu \in B_{r_{\lambda}}(\lambda)$,
\begin{equation}\label{Rn.ineq.4}
\|R_k(\mu) -R_k(\lm) \| \leq 
\frac{|\mu -\lm| \|R_k(\lm)\|^2}{1 - |\mu -\lm| \|R_k(\lm)\|}
\leq \frac{r_{\lambda} M_\lm^2}{1 - r_{\lambda} M_\lm}.
\end{equation}
There exists $r_{\lambda}>0$ sufficiently small such that the r.h.s.\ of~\eqref{Rn.ineq.4} is less than $\alpha/2$.
Now we estimate, for all $\mu \in B_{r_{\lambda}}(\lm)$ and all $k\geq k_{\lm}$,
\begin{equation}
\begin{aligned}
\|R_k(\mu)\| 
& \geq 
\|R(\lm)\| - \|R_k(\lm)P_k - R(\lm)P\| - \|R_k(\mu) - R_k(\lm)\|
\\
&> \Big(\frac{1}{\varepsilon} + \alpha\Big) - \frac{\alpha}{2} - \frac{\alpha}{2}=\frac{1}{\varepsilon}.
\end{aligned}
\end{equation}
\end{proof}

In the following lemma we use the notation $\omega_{\delta}(\Omega):= \{ z \in \C \, : \, \dist(z,\Omega) < \delta\}$, $\overline{\omega_{\delta}}(\Omega) := \overline{\omega_{\delta}(\Omega)}$ for the open and closed $\delta$-neighbourhood of a set $\Omega \subset \C$.

\begin{lemma}\label{lem.K.not}
Let assumptions \ref{ass.K}--\ref{ass.1.3} of Theorem~{\rm\ref{thm.psp.conv}} be satisfied.
Define 
\begin{equation}
\Lambda:=\overline{\sigma_\varepsilon(T)} \cap K, \quad \Lambda_k:=\overline{\sigma_\varepsilon(T_k)} \cap K.
 \end{equation}
Then for every $\delta >0$ there exists $k_\delta \in\N$ such that 
\begin{equation}\label{K.incl}
\Lambda_k \subset \overline{\omega_\delta}(\Lambda), \quad \Lambda \subset \overline{\omega_\delta}(\Lambda_k), \quad k\geq k_\delta.
\end{equation}
\end{lemma}
\begin{proof}
With regard to Lemma \ref{lem.K}, we need to consider only compact sets $K\subset\C$ that are not subsets of $\sigma_\varepsilon(T)$.

We start with the proof of the first inclusion in~\eqref{K.incl}. 
Let $\delta >0$. 
For $\lm \in K \setminus \omega_\delta(\Lambda)$ holds $\lm \in K \setminus \Lambda$, hence 
\begin{equation}
\lm \notin \overline{\sigma_\varepsilon(T)} = \left\{ z \in \C \, : \, \|R(z)\| \geq \frac{1}{\varepsilon} \right\},
\end{equation}
where the equality is guaranteed by the assumption \ref{ass.const}. Therefore
\begin{equation}\label{eq.estimate.less.than.eps}
\|R(\lm) \| < \frac{1}{\varepsilon}, \quad \lm\in K\backslash\omega_{\delta}(\Lambda).
\end{equation}
 
We show that there exists $k_\delta\in\N$ such that $\|R_k(\lm)\| < 1/\varepsilon$  for all $k \geq k_\delta$ and all $\lm \in K \setminus \omega_\delta(\Lambda)$.
As a consequence, for all $k \geq k_\delta$, $\overline{\sigma_\varepsilon(T_k)} \cap \left(K \setminus \omega_\delta(\Lambda)\right) = \emptyset$, and so the first inclusion in~\eqref{K.incl} follows from 
\begin{equation}
\begin{aligned}
\Lambda_k
 \subset 
\omega_\delta(\Lambda) 
\cup 
\left(\overline{\sigma_\varepsilon(T_k)} \cap \left(K \setminus \omega_\delta(\Lambda)\right)\right) 
 =
\omega_\delta(\Lambda)
\subset 
\overline{\omega_\delta} (\Lambda).
\end{aligned}
\end{equation}
The existence of such a $k_\delta\in\N$ can be justified in the following way.

Let $ \lm\in K\backslash\omega_{\delta}(\Lambda)$.
It follows from~\eqref{eq.estimate.less.than.eps} together with $T_k\gnrlong T$ and Lemma~\ref{lem.rb}~ii) that
there exist $k_\lm\in\N$ and $\alpha>0$ such that $\|R_k(\lm)\| \leq 1/\varepsilon - \alpha$ for all $k\geq k_\lm$. 
Moreover, by proceeding analogously as in~\eqref{Rn.ineq.4}, we can verify that there exists $r_{\lm}>0$ such that, for all $k\geq k_{\lm}$, we have $B_{r_{\lm}}(\lm) \subset \rho(T_k)$,
and 
$\|R_k(\lm) - R_k(\mu)\| < \alpha$ for all $\mu \in B_{r_{\lm}}(\lm)$. Therefore, 
for all $ \mu \in B_{r_{\lm}}(\lm)$, we have
\begin{equation}
\|R_k(\mu)\| \leq \|R_k(\lm)\|  + \|R_k(\lm) - R_k(\mu)\| <\frac{1}{\varepsilon}, \quad k\geq k_{\lm}.
\end{equation}
The existence of the desired $k_\delta\in\N$ now follows from the compactness of~$K \setminus \omega_\delta(\Lambda)$.

The second inclusion in~\eqref{K.incl}  is proved by contradiction. Assume that there exist $\delta>0$, an infinite subset $I_1\subset\N$ and $\{\lm_k\}_{k\in I_1}\subset \Lambda$ such that 
$\lm_k \notin \overline{\omega_\delta} (\Lambda_k)$, $k \in I_1$. Since $\Lambda$ is compact, there exist $\lambda_0\in\Lambda$ and an infinite subset $I_2\subset I_1$ such that $\lambda_k \to \lambda_0$ for $k \in I_2$. 
Moreover, since $\Lambda=\overline{\sigma_\varepsilon(T)} \cap   K = \overline{\sigma_\varepsilon(T) \cap K}$ by assumption~\ref{ass.K}, there exists $\widetilde \lm_0 \in \sigma_\varepsilon(T) \cap K $ such that $|\lm_0 - \widetilde \lm_0| < \delta /2$. 
By Lemma~\ref{lem.K}, there exists $k_0\in\N$ such that $\widetilde \lm_0 \in \sigma_\varepsilon (T_k) \cap K\subset\Lambda_k$ for all $k\geq k_0$. However, 
\begin{equation}
|\lm_k-\widetilde \lm_0| \leq |\lm_k - \lm_0| + |\lm_0 - \widetilde \lm_0| < \delta
\end{equation} 
for all sufficiently large $k \in I_2$; this is a contradiction to 
$\lm_k \notin \overline{\omega_\delta} (\Lambda_k)$.
\end{proof}

\begin{proof}[Proof of Theorem~{\rm\ref{thm.psp.conv}}]
Take some arbitrary $\delta>0$. Then, using Lemma \ref{lem.K.not}, we obtain for all~$k\geq k_\delta$:
\begin{equation}
\begin{aligned}
&d_{\rm H} \left(
\overline{\sigma_{\varepsilon}(T_k)} \cap K, \overline{\sigma_{\varepsilon}(T)} \cap K
\right) \\
&=d_{\rm H} (\Lambda_k, \Lambda)
=
\max \left\{
\sup_{z \in \Lambda_k } \dist(z,\Lambda), 
\sup_{z \in \Lambda} \dist(z,\Lambda_k)
\right\}
\\
&\leq
\max \left\{
\sup_{z \in \overline{\omega_\delta}(\Lambda)} \dist(z,\Lambda),\sup_{z \in \overline{\omega_\delta}(\Lambda_k)} \dist(z,\Lambda_k)
\right\} \leq \delta.
\qedhere
\end{aligned}
\end{equation}
\end{proof}

 \begin{remark}\label{rem.Banach.2}
The Hilbert space structure is used only in the proof of Lemma~\ref{lem.rb}.
If $T$ and $T_k$ act in the same Banach space $\X$, then claim b) in the proof of Lemma~\ref{lem.rb} can be proved using the splitting
$\sigma(T)=\sigma_{\rm app}(T) \cup \sigma_{\rm app}(T^*)$, where $T^*$ is
the adjoint operator in the Banach space $\X^*$.
Hence Theorem~\ref{thm.psp.conv} remains valid for operators that act in the same Banach space and converge in the norm resolvent sense.
 \end{remark}

We indicate the different and additional steps that are needed to prove Theorem~\ref{thm.psp.conv.n} for an arbitrary $n>0$.
In Lemmas~\ref{lem.rb}--\ref{lem.K} and their proofs, we replace all appearing pseudospectra by $n$-pseudospectra and all resolvents by their $2^n$-th powers. 
In addition, we make the following modifications:
\begin{enumerate}[label=\rm{\roman{*})}]
\item For an arbitrary $n\in\N$, we define a notion that is analogous to the region of boundedness, \cf~\eqref{eq.rb.n}.
In Lemma~\ref{lem.rb.all.equal}~i) below, it is shown that if $T_k\gnrlong T$, then this notion coincides, for any $n\in\N$, with the region of boundedness (and thus with the resolvent set of $T$).
Moreover, for every $\lm\in\rho(T)$, we have norm convergence of the powers of the resolvents, \cf~Lemma~\ref{lem.rb.all.equal}~ii).
\item The resolvent expansions used in~\eqref{Rn.ineq.3} and~\eqref{Rn.ineq.4} are generalised in Lemma~\ref{lem.res.expansion.n} below; 
instead of~\eqref{Rn.ineq.3}, we use Lemma~\ref{lem.res.expansion.n} with $\nu:=\lm$, and instead of~\eqref{Rn.ineq.4}, we use Lemma~\ref{lem.res.expansion.n} with $\nu:=\mu$, $\lm_k:=\lm$.
\item The estimate in~\eqref{Rn.ineq.2} is replaced by 
$\left\|R_k(\lm_k)^{2^n}\right\|\geq |\lm-\lm_k|^{-2^n}$, $k\in I_1$.
\end{enumerate}

The following lemma is used as a tool in the subsequent Lemmas; it is proved by induction over $l$ and with the first resolvent identity.
\begin{lemma}\label{lem.res.powers}
Let $l\geq 2$.
Then, for every $k\in\N$ and all $\lm,\lm_0\in\rho(T_k)$,
\begin{equation}
\begin{aligned}
R_k(\lm)=&\sum_{j=1}^{l-1} (\lm-\lm_0)^{j-1} \,R_k(\lm_0)^j\\ &+ (\lm-\lm_0)^{l-1} \,\left(I-(\lm-\lm_0)R_k(\lm_0)\right)^{l-1}R_k(\lm)^l.
\end{aligned}
\end{equation}
\end{lemma}

Let $n\in\N$. Define a generalisation of the region of boundedness by
\begin{equation}\label{eq.rb.n}
\begin{aligned}
\Delta_{\rm b}^{(n)}(\{T_k\}_k )
:=\Big\{
\lambda \in\C:\,\exists\, k_\lm \in \N, \, \exists\, M_\lm^{(n)}>0,\,  \forall \, k \geq k_\lm: \, 
\\
\lambda\in\rho(T_k) \mbox{ and }\left\|R_k(\lambda)^{2^n}\right\| \leq M_\lm^{(n)}
\Big\}.
\end{aligned}
\end{equation}

\begin{lemma}\label{lem.rb.all.equal}
Let $n\in\N$ and let $T, T_k$ be densely defined. If $T_k \gnrlong T$, then
\begin{enumerate}[label=\rm{\roman{*})}]
\item $\Delta_{\rm b}^{(n)}\left(\{T_k\}_k \right) =\Delta_{\rm b}\left(\{T_k\}_k\right)= \rho(T)$;
\item for all $\lm\in\rho(T)$,
\begin{equation}\label{eq.gnr.forall.lm.n}
\left\|R_k(\lm)^{2^n} P_k - R(\lm)^{2^n}P \right\|\tolong~0, \quad k\to\infty.
\end{equation}
\end{enumerate}
\end{lemma}

\begin{proof}
\begin{enumerate}[label=\rm{\roman{*})}]
\item 
The inclusion $\Delta_{\rm b}\left(\{T_k\}_k\right)\subset\Delta_{\rm b}^{(n)}\left(\{T_k\}_k \right)$ is implied by the estimate $\left\|R_k(\lm)^{2^n}\right\|\leq\|R_k(\lm)\|^{2^n}$.
For the opposite direction we use $\lm_0\in \cap_{k\in\N} \rho(T_k)\cap\rho(T)$ such that $\|R_k(\lm_0)\|$, $k\in\N$, are uniformly bounded; such a $\lm_0$ exists by the assumption $T_k\gnrlong T$.
Now Lemma~\ref{lem.res.powers} with $l:=2^n$ implies that every $\lm\in \Delta_{\rm b}^{(n)}\left(\{T_k\}_k\right)$ belongs to $\Delta_{\rm b}\left(\{T_k\}_k\right)$.
\item
The claim is an immediate consequence of Lemma~\ref{lem.rb}~ii) and 
\begin{equation}
R(\lm)^{2^n}P=\left(R(\lm)P\right)^{2^n},\quad 
R_k(\lm)^{2^n}P_k=\left(R_k(\lm)P_k\right)^{2^n}. 
\qedhere
\end{equation}
\end{enumerate}
\end{proof}

\begin{lemma}\label{lem.res.expansion.n}
Let $n\in\N$ and $T_k\gnrlong T$.
Let $\{\lm_k\}_k\subset\C$ be a bounded sequence such that $\|R_k(\lm_k)\|$ or $\left\|R_k(\lm_k)^{2^n}\right\|$ is uniformly bounded in $k$.
Then there exists $C>0$ such that  every $\nu\in\C$ with $|\nu-\lm_k|<1/C$, $k\in\N$, satisfies
\begin{equation}
\left\|R_k(\nu)^{2^n}-R_k(\lm_k)^{2^n}\right\|\leq\frac{C^{2^n}}{(1-|\nu-\lm_k| C)^{2^n}}\sum_{j=1}^{2^n}\begin{pmatrix}2^n\\ j\end{pmatrix}C^j |\nu-\lm_k|^j, \quad k\in\N.
\end{equation}
\end{lemma}

\begin{proof}
Since $T_k\gnrlong T$, there exists $\lm_0\in\cap_{k\in\N}\rho(T_k)\cap\rho(T)$ such that $\|R_k(\lm_0)\|$, $k\in\N$, are uniformly bounded.
The first resolvent identity and binomial theorem yield the expansion
\begin{equation}
\begin{aligned}
&R_k(\nu)^{2^n}-R_k(\lm_k)^{2^n}\\
&=
-R_k(\lm_k)^{2^n} ( I - (\nu-\lm_k)R_k(\lm_k))^{-2^n}\,\sum_{j=1}^{2^n}\begin{pmatrix}2^n\\ j\end{pmatrix}R_k(\lm_k)^j (\lm_k-\nu)^j.
\end{aligned}
\end{equation}
If $\|R_k(\lm_k)\|$ is uniformly bounded in $k$, then the claim follows by setting $C:=\sup_k\|R_k(\lm_k)\|$.
If, on the other hand, $\left\|R_k(\lm_k)^{2^n}\right\|$ is uniformly bounded in $k$, then the boundedness of $\{\lm_k\}_k$ and Lemma~\ref{lem.res.powers} yield a uniform bound for $\|R_k(\lm_k)\|$.
\end{proof}

{\footnotesize
\bibliographystyle{acm}
\bibliography{references}
}

\end{document}